\documentclass[11 pt]{amsart}
\usepackage{amssymb}
\usepackage{amsmath}
\usepackage{amsthm}
\usepackage{pxfonts}
\usepackage{algorithmic}
\usepackage{enumerate}
\usepackage{booktabs}
\usepackage{caption}
\usepackage{placeins}
\usepackage{xcolor}
\usepackage{bbm}
\usepackage{color}
\usepackage{fullpage}
\definecolor{mylinkcolor}{rgb}{0.5,0.0,0.0}
\definecolor{myurlcolor}{rgb}{0.0,0.0,0.75}
\usepackage{lscape}
\usepackage{rotating}
\usepackage[
pdfauthor={},
pdfkeywords={},
pdftitle={},
pdfcreator={},
pdfproducer={},
linktocpage,colorlinks,bookmarksnumbered,linkcolor=blue,
citecolor=red,urlcolor=red]{hyperref}

\definecolor{lime}{HTML}{A6CE39}
\DeclareRobustCommand{\orcidicon}{%
	\begin{tikzpicture}
		\draw[lime, fill=lime] (0,0) 
		circle [radius=0.16] 
		node[white] {{\fontfamily{qag}\selectfont \tiny ID}};
		\draw[white, fill=white] (-0.0625,0.095) 
		circle [radius=0.007];
	\end{tikzpicture}
	\hspace{-2mm}
}

\usepackage{graphicx,pgfarrows,pgfnodes}
\usepackage[mathscr]{eucal}
\usepackage{verbatim}
\usepackage{url,cite,mleftright}
\usepackage{tkz-fct}
\usepackage{tikz}

\usetikzlibrary{shapes.geometric}
\usepackage[utf8]{inputenc}
\usepackage{mathtools}
\usepackage{url}
\usepackage{scalefnt}
\DeclareMathOperator{\co}{\operatorname{codim}}

\usepackage[utf8]{inputenc}

\newcommand{\QED}{\hspace{\stretch{1}} $\blacksquare$}

\newcommand{\CC}{\mathbb{C}}

\newcommand{\NN}{\mathbb{N}}

\newcommand{\RR}{\mathbb{R}}
\newcommand{\ZZ}{\mathbb{Z}}

\newcommand{\la}{\langle}
\newcommand{\ra}{\rangle}

\theoremstyle{plain}
\newtheorem{thm}{Theorem}

\newtheorem{conj}[thm]{Conjecture}
\newtheorem{prop}[thm]{Proposition}

\theoremstyle{definition}

\theoremstyle{remark}
\newtheorem{rem}[thm]{Remark}
\newtheorem*{ex}{Example}

\allowdisplaybreaks

\numberwithin{equation}{section}
\numberwithin{thm}{section}
\allowdisplaybreaks

\numberwithin{equation}{section}
\numberwithin{thm}{section}

\usepackage{academicons}
\usepackage{xcolor}

\calclayout

\newcommand{\orcidParth}{\href{https://orcid.org/0009-0005-3186-5984}{\orcidicon}}

\newcommand{\orcid}[1]{\href{https://orcid.org/#1}{\textcolor[HTML]{A6CE39}{\aiOrcid}}}

\begin{document}
	\title{Counting ideals in numerical semigroups}
	\author{Parth Chavan \orcidParth}
	\thanks{The author was supported by The 2022 Spirit of Ramanujan Fellowship and The 2022 Mehta Fellowship}
	\email{spc2005@outlook.com}
	\date{\today}
	\keywords{Numerical semigroup, Ideal, Catalan number}
	\subjclass{20M14, 05A19.}
	\maketitle
	\begin{abstract} If $S$ is a numerical semigroup, let $m(S,k)$ denote the number of ideals of $S$ with codimension $k$ and let $n(S,k)$ denote the number of ideals of $S$ with conductor $k$. We compute the generating function of the sequence $m(S,k)$ for all numerical semigroups of embedding dimension $2$ and for $S = \langle 3,n+2,2n+1\rangle$. We also prove that the sequence $n(S,k)$ becomes stationary after a certain term and compute the stationary terms for numerical semigroups of the form $\langle a,a+1 \rangle$.
 %a set $I \subset S$ is called an ideal if $I+S \subseteq S$. Let $m(S,k)$ and $n(S,k)$ denote the number of numerical semigroups $S' \subseteq S$ such that $F(S')=k$ and $g(S')=g(S)+k$ respectively. In this work, we prove that the sequence $m(S,k)$ eventually becomes stationary for all numerical semigroups $S \neq \NN$ and explicitly compute the stationary terms for numerical semigroups of the form $\langle a,a+1 \rangle$. We also explicitly compute the numbers $n(S,k)$ for numerical semigroups with embedding dimension $2$ and certain families of numerical semigroups with embedding dimension $3$ thereby proving a recent conjecture of Moreno-Frias and Rosales for these numerical semigroups. 
	\end{abstract}
	
	\section{Introduction}
Let $\NN$ denote the set of non-negative integers. A numerical semigroup is a subset $S$ of $\NN$ that is closed under addition, $0 \in S$, and $\NN \setminus S$ has finitely many elements. For any $a_1,\ldots,a_k \in \NN$, let $\langle a_1,\ldots,a_k \rangle = \{a_1n_1+\ldots+a_kn_k \mid n_1,\ldots,n_k \in \NN\}$. The set $\langle a_1,\ldots,a_k \rangle$ is a numerical semigroup if and only if $\gcd{a_1,\ldots,a_k} =1$. Moreover, any numerical semigroup admits a minimal system of generators $\{a_1<\ldots<a_p\}$ such that $\langle a_1,\ldots,a_p \rangle = S$. The number $p$ is called the \emph{embedding dimension} of $S$ and is denoted by $e(S)$.

Given a set $P \subseteq \NN$ and an integer $n \in \NN$, let $n+P = \{n+p \mid p \in P\}$, $\NN_{\geqslant n} = \{i \in \NN \mid i \geqslant n\}$ and $[n] = \{0,1,\ldots,n\}$. An \emph{ideal} is a subset of a numerical semigroup that has an additive structure over it: a non-empty set $I \subseteq S$ is called an \emph{ideal} of $S$ if $I+S = I$. For a set $S$, we represent its cardinality by $|S|$. For an ideal $I \subseteq S$, the \emph{codimension} and \emph{conductor}  of $I$ are defined as $\operatorname{codim}(I)= |S \setminus I|$ and  $c(I)=\min\{n \in \NN \mid \NN_{\geq n} \subseteq I\}$ respectively. Both of these are equality reversing invariants. In other words, ideals $I_1 \subseteq I_2$ satisfy $\operatorname{codim}(I_2) \leq \operatorname{codim}(I_1)$ and $c(I_2) \leq c(I_1)$. Let $I(S)$ denote all ideals of $S$. An equivalent form of a recent conjecture posed by Moreno-Frias and Rosales is stated below

\begin{conj}\label{conjros}
		Given a numerical semigroup $S \neq \NN$, let $m(S,k) = |\{I \in I(S) \mid \co(I)=k \}|$. Then there exists a constant $m_S$, such that $i(S,k) \leqslant i(S,k+1)$ for all $0<k<m_S$ and $i(S,m_S) = i(S,m_S+k)$ for all $k \in \NN$.
	\end{conj} 
	They prove this conjecture for ordinary numerical semigroups, $\{0\}\cup \NN_{\geq k}$, and numerical semigroups of the form $\langle 2,2k+1 \rangle$ \cite{RosaMor,Rosa2}. The aim of this article is to prove this conjecture for all numerical semigroups $S$ with $e(S)=2$ or $S= \langle 3,n+2,2n+1 \rangle$ for some $n \in \NN$.

  Another question of interest is to determine whether the sequence $n(S,k) = |\{I \in I(S) \mid c(I)=k \}|$ satisfies similar properties. We answer this question affirmatively by proving that $n(S,k)$ eventually becomes stationary and also explicitly compute these stationary terms for numerical semigroups of the form $\langle a,a+1 \rangle$.

\section{Preliminaries}

This section is devoted to proving that the sequences $m(S,k), n(S,k)$ become stationary. In what follows, let $c(S)-1 = F(S)$. Moreno-Frias and Rosales, via the following result, prove that the sequence $m(S,k)$ becomes stationary after $k$ exceeds $F(S)$.
	
	\begin{thm}[\cite{RosaMor}]\label{thmros}
		The following inequalities hold
		\begin{enumerate}
			\item $m(S,F(S))=m(S,F(S)+k)$ for all $k \in \NN$
			\item $m(S,F(S)-1) < m(S,F(S))$.
		\end{enumerate}
	\end{thm}

 We now prove $n(S,k)$ becomes stationary for any numerical semigroup $S$.
\begin{prop}\label{ms} Let $\mathcal{A}(S,k) = \{ I \in I(S) \mid c(I) = k\}$. Then for all $k \in \NN_{k \geqslant 2}$, the map $$F: \mathcal{A}\left(S, 2F(S)+2\right) \to \mathcal{A}\left(S, 2F(S)+k+2\right), $$ defined by $F(I)=I+k$ is a bijection.
	\end{prop}
	\begin{proof}
		We first prove that the map $F$ is well-defined. For any ideal $I \in \mathcal{A}(S, 2F(S)+k)$, let $m(I) = \min \{i \in I\}$. If $m(I)<F(S)+k-1$, then $c(I) \leqslant c(m(I)+S)=m(I)+ F(S)+1$ since $m(I)+S \subseteq I$. This contradicts our assumption since it implies $c(I) < 2F(S)+k$. Thus any ideal $I \in \mathcal{A}(S, 2F(S)+2)$ satisfies $I \subset \NN_{\geqslant F(S)+1}$. For any two integers $i \in I , s \in S$, we have the inclusion $i+s+k \in I+k \subset \NN_{F(S)+k+1}$ since $i+s \in I$. Thus the set $I+k$ is an ideal. Since $I \subset \NN_{\geqslant F(S)+1}$, it's clear that $c(I+k)=c(I)+k$. Thus $F$ is well defined. 
		
		Now consider the map $G: \mathcal{A}(S, 2F(S)+k+2) \to \mathcal{A}(S,2F(s)+2)$ defined by $G(I)=I-k$. Using a similar method as before it can be proved that the map $G$ is well-defined as well. The maps $F,G$ are clearly injective which thereby establishes a bijection.
	\end{proof}
	Let $n_S$ be the smallest integer such that $n(S,n_S)=n(S,m)$ for all $m \in \NN_{\geq n_S+1}$.  Proposition \ref{ms} implies that $n_S \leqslant 2F(S)+2$.
		\section{Counting Ideals with Respect to Codimension}
	
	\subsection{Ideal generating function} The \emph{ideal generating function} corresponding to the numerical semigroup $S$ is defined as
	\begin{equation}\label{igf}
		I\left(S;q\right) = \sum_{I\, \in \, I\left(S\right)} q^{\operatorname{codim}(I)}.
	\end{equation}
	If $S$ has set of generators $\{a_1,\ldots,a_r\}$, we occasionally use the notation $I(a_1,\ldots,a_r;q)$ for $I\left(S;q\right)$. Let $R=k[t^{a_1},\ldots,t^{a_r}]$ be the \emph{semigroup ring} over a field $k$ associated with the semigroup $S=\la a_1,\ldots,a_r \ra$ and let $J\left(a_1,\ldots,a_r\right)$ be the set of all monomial ideals in this ring. The ideal generating function can alternatively be defined as 
	\begin{equation}\label{firsteqn}
		I\left(a_1,\ldots,a_r,q\right)=\sum_{J \,\in\, J\left(a_1,\ldots,a_r\right)} q^{\dim_{k}(R/J)}.
	\end{equation}
	For a rational function $\frac{p(x)}{q(x)} \in \RR(x)$, we define its degree as $\deg\left(p(x)\right)-\deg\left(q(x)\right)$.
	It turns out that the ideal generating function is a rational function as described next. 
	\begin{prop}\label{0}
		The ideal generating function of a numerical semigroup $S$ can be written as 
		\[I\left(S;q\right) = \frac{f(q)}{1-q},\]
		where $f(q) \in \ZZ[q]$. Moreover, we have \begin{equation}\label{FRO}
			 \deg\left(I\left(S;q\right)\right) = \deg\left(f(q)\right)-1=F(S).	\nonumber\end{equation}
	\end{prop}
	
	\begin{proof}
		
		We use Theorem \ref{thmros} to rewrite the ideal generating function as 
		\begin{equation}\label{igf}
			I\left(a_1,\ldots,a_r;q\right) = \sum_{\co\left(I\right) < F\left(S\right) } q^{\co\left(I\right)} + \frac{m\left(S,F(S)\right)q^{F(S)}}{1-q}.
		\end{equation}
		Thus, we get $I(a_1,\ldots,a_r;q) = \frac{f(q)}{1-q}$ for some polynomial $f(q) \in \ZZ[q]$. Since $m(S,F(S)-1) < m(S,F(S))$, we deduce that $\deg\left(f(q)\right) - 1 = F\left(S\right)$.
	\end{proof}
	
	\section{Computing $I(a,b;q)$}\label{S3}
	
	In this section, we compute the ideal generating function of numerical semigroups $S$ with embedding dimension two. Throughout, we assume $S= \la a,b \ra$ where $\gcd(a,b)=1$ and $a<b$. 
	
	We first associate a tabular representation to $\langle a,b \rangle$. Let $R\left(a\right)$ be the rectangle having $a$ rows and an infinite number of columns. Label the bottom left corner of $R\left(a\right)$ as $0$. We now label the remaining square in $R\left(a\right)$ with integers so that a shift by $1$ up adds $b$ and a shift by $1$ to the right adds $a$. As a result, a square is labelled by $ax+by$ where $\left(x,y\right)$ are the coordinates of the bottom left vertex of the square. Let $R\left(a,b\right)$ denote the grid we obtain after labelling $R\left(a\right)$ in this way. For instance, Figure \ref{fig:fig:1} shows $R\left(3,4\right)$. Notice that each element $s \in \la a,b \ra$ occurs exactly once as a label in $R\left(a,b\right)$.

	\begin{figure}
		
		\begin{center}
			
			\begin{tikzpicture}
				\foreach \i in {1,...,3,4,5,6,7} {
					\foreach \j in {1,...,3} {
						\draw (\i, \j) rectangle (\i+1, \j+1);
					}
				}
				\node at (1.5,1.5) {$0$};
				\node at (2.5,1.5) {$3$};
				\node at (3.5,1.5) {$6$};
				\node at (4.5,1.5) {$9$};
				\node at (1.5,2.5) {$4$};
				\node at (1.5,3.5) {$8$};
				\node at (2.5,2.5) {$7$};
				\node at (3.5,2.5) {$10$};
				\node at (4.5,2.5) {$13$};
				\node at (2.5,3.5) {$11$};
				\node at (3.5,3.5) {$14$};
				\node at (4.5,3.5) {$17$};
				\node at (5.5,1.5) {$12$};
				\node at (5.5,2.5) {$16$};
				\node at (5.5,3.5) {$20$};
				\node at (6.5,1.5) {$15$};
				\node at (6.5,2.5) {$19$};
				\node at (6.5,3.5) {$23$};
				\node at (7.5,1.5) {$\ldots$};
				\node at (7.5,2.5) {$\ldots$};
				\node at (7.5,3.5) {$\ldots$};
			\end{tikzpicture}
		\end{center}
		
		\caption{Labelling of $R(3)$.}
		\label{fig:fig:1}
	\end{figure}
	
	Let $P(a,b)$ be the set of lattice paths $P$ that uses the steps $(1,0)$, $(0,-1)$ and satisfy the following conditions:
	\begin{enumerate}
		\item $P$ starts at $(0,a)$ and ends at $(x,0)$ for some integer $x \in \NN$.
		\item $P$ ends with a step of the form $(0,-1)$.
		\item The number of steps of the form $(1,0)$ between the first and last step of form $(0,-1)$ in $P$ is at most $b$.
	\end{enumerate}
The lattice paths in $P(a,b)$ are also called staircases. It is known that staircases are in one-to-one correspondence with monomial ideals in $\CC[x^a,x^b]$ \cite{hilbn, bria}. This can be restated via the following result:
	\begin{prop}\label{rcat}
		Lattice paths in $P(a,b)$ are in one-to-one correspondence with ideals in $I\left(\la a,b \ra\right)$.
	\end{prop}
This can be proved by mapping a lattice path $P \in P(a,b)$ to the set of elements in $R(a,b)$ that lie to the right of $P$. Requirements on $P$ imply that this set is an ideal. Any ideal can be uniquely mapped to a lattice path in $P(a,b)$ that traces the left edges of the smallest labels in each row of $R(a,b)$ that also lie in the ideal. The codimension of an ideal equals the number of squares under its corresponding path.
	\begin{ex}
		Figure \ref{fig:2} shows $R(3,4)$ and an example of a path that corresponds to an ideal. In Figure \ref{fig:2}, let the dark lines denote path $P$. The elements to the right of $P$ clearly form an ideal. 
		\begin{figure}

			\begin{center}
				\begin{tikzpicture}
					\foreach \i in {1,...,3,4,5,6,7} {
						\foreach \j in {1,...,3} {
							\draw (\i, \j) rectangle (\i+1, \j+1);
						}
					}
					\node at (1.5,1.5) {$0$};
					\node at (2.5,1.5) {$3$};
					\node at (3.5,1.5) {$6$};
					\node at (4.5,1.5) {$9$};
					\node at (1.5,2.5) {$4$};
					\node at (1.5,3.5) {$8$};
					\node at (2.5,2.5) {$7$};
					\node at (3.5,2.5) {$10$};
					\node at (4.5,2.5) {$13$};
					\node at (2.5,3.5) {$11$};
					\node at (3.5,3.5) {$14$};
					\node at (4.5,3.5) {$17$};
					\node at (5.5,1.5) {$12$};
					\node at (5.5,2.5) {$16$};
					\node at (5.5,3.5) {$20$};
					\node at (6.5,1.5) {$15$};
					\node at (6.5,2.5) {$19$};
					\node at (6.5,3.5) {$23$};
					\node at (7.5,1.5) {$\ldots$};
					\node at (7.5,2.5) {$\ldots$};
					\node at (7.5,3.5) {$\ldots$};
					\draw [line width=0.7mm, black ] (1,4) -- (2,4);
					\draw [line width=0.7mm, black ] (2,4) -- (2,3);
					\draw [line width=0.7mm, black ] (2,3) -- (3,3);
					\draw [line width=0.7mm, black ] (3,3) -- (4,3);
					\draw [line width=0.7mm, black ] (4,3) -- (4,2);
					\draw [line width=0.7mm, black ] (4,2) -- (5,2);
					\draw [line width=0.7mm, black ] (5,2) -- (5,1);
				\end{tikzpicture}
				
			\end{center}
			\caption{Tabular representation of the ideal $\{11,12,13\}+\la 3,4 \ra$.}
			\label{fig:2}
		\end{figure}
	\end{ex}

	We can map any path $P \in P(a,b)$ to a word $w(P)$ over the alphabet $\{D,R\}$ by replacing the steps $(0,-1), (1,0)$ by $D,R$ respectively. Let $W(a,b)$ be the set words corresponding to paths in $P(a,b)$. Let $w_i$ denote the letter at the $i^{\text{th}}$ place of a word $w$. An \emph{inversion} is a pair $(i,j)$ where $i<j$ and $w_i= R , w_j= D$. Notice that the area under a lattice path $P$ equals the number of inversions in the word $w(P)$. Let $\operatorname{inv}(w)$ denote the number inversions in $w$ and let $Q(n,k)$ denote the set of all words over the alphabet $\{D,R\}$ that have $n,k$ occurrences of the letters $D,R$ respectively. The following result holds \cite{Geo}
	\begin{equation}\label{eq1}
		\binom{n}{k}_{q} = \sum_{w \,\in\, Q(n,k)} q^{\operatorname{inv}(w)},
	\end{equation}
	where 
	\[\binom{n}{k}_{q} = \frac{\prod_{i=1}^{n}(1-q^i)}{\prod_{i=1}^{k}(1-q^i)\prod_{i=1}^{n-k}(1-q^i)}.\]
	Now we are ready to explicitly evaluate $I\left(a,b;q\right)$.
	\begin{thm}\label{I(a,b)}
		The ideal generating function of the numerical semigroup $ \langle a,b \rangle$ is given by
		\[I(a,b;q)=\frac{1}{1-q^a}\binom{a+b-1}{a-1}_{q}.\]
	\end{thm}
	\begin{proof}
		 Let $S_i\left(a,b\right)$ be the set of paths in $P\left(a,b\right)$ such that the first $(i-1)$ steps are $(1,0)$ and the $(i)^{\text{th}}$ step is $(0,-1)$. Let 
		\[S_i(q) = \sum_{P \in S_i(a,b)} q^{s(P)},\]
		where $s(P)$ denotes the number of squares under $P$. Clearly, we have $S_{i}(q) = q^{a\left(i-1\right)} S_1(q)$. Therefore we deduce
		\[I\left(a,b;q\right) = \frac{S_1 (q)}{1-q^a}.\]
		Any path $P \in S_1(a,b)$ can be uniquely mapped to a word in $Q(a-1,b)$ by first mapping the path $P$ to the word $w(P)$ and then appending $w(P)$  with a suitable number of $R$ so that the total number of occurrences of the letter $R$ is $b$. From equation \eqref{eq1}, we have
		\[S_1(q) = \binom{a+b-1}{a-1}_{q},\]
		thereby proving the desired result.
	\end{proof}

	Using Proposition \ref{0} and Theorem \ref{I(a,b)} we have 
	\begin{equation}\label{rem}
		m(S,F(a,b)) = \lim_{q \to 1} (1-q)I(a,b;q) = \frac{1}{a+b}\binom{a+b}{a}.
	\end{equation}
	This result has been derived in multiple ways. See \cite{w1,w2,w3,w4} for instance. In particular,  Moyano-Fernández and Uliczka \cite{w4} derive this result by enumerating semi-modules. They provide a bijection between isomorphism classes of semimodules of $\langle a,b \rangle$ and the set lattice paths from $(0,a)$ to $(b,0)$ not crossing the diagonal. This proves equation \eqref{rem} since the number of isomorphism classes of semimodules equals $m(\langle a,b \rangle, ab-a-b)$.  They also provide a formula for the number of isomorphism classes of semimodules of $\langle a,b \rangle$ that are minimally generated by $r$ elements. However, even though lattice paths are at the heart of both methods, they differ in that our method relies on counting ideals with respect to codimension, whereas their method relies on counting the number of isomorphism classes of semimodules with respect to the minimal number of generators.

 A simple calculation yields $\deg\left(I\left(a,b;q\right)\right)=ab-a-b$. Thus using Proposition \ref{0}, we rederive Sylvester's theorem which states $F(a,b)=ab-a-b$. 
	Moreover, Proposition \ref{0} also implies that 
	\[\binom{a+b-1}{a-1}_{q} = (1+q+\ldots+q^a)f(q),\]
	for some polynomial $f(q) \in \ZZ[q]$.
	Since $q$-binomial coefficients are unimodal \cite{unimod} and have positive integer coefficients, we deduce that $f(q) \in \NN[q]$. This implies Conjecture \ref{conjros} is true for all numerical semigroups with embedding dimension $2$. 
	\section{Computing $I(a,b,c;q)$}\label{S4}
	
	The main goal of this section is to compute the ideal generating function of a certain family of numerical semigroups with embedding dimension three. In particular, we compute $I(a,b,c;q)$ when $a<b<c$ are mutually coprime positive integers such that $a \mid b+c$, by using a tabular representation similar to the one used in the previous section. In what follows, we assume that $a,b,c$ satisfy these requirements. To provide a tabular representation for $\la a,b,c \ra$, we first need the following result.
	
	\begin{prop}\label{uniq}
		Given a semigroup of the form $\la a,b,c \ra$ with $a\mid b+c$, there exists a positive integer $k_1$ such that any $s \in \la a,b,c \ra$ can be uniquely written as either
		$a x + by$ or $ ax+cz$ where $x \in \NN, y \in [k_1-1] $ and, $ z \in [a-k_1] \setminus \{0\}$. 
	\end{prop}
	\begin{proof}  The proof is left as an exercise for the reader \end{proof}
	%		Let  $k_1$ and $k_2$ be the least positive integers such that $b k_1 \in  \la a,c \ra , c k_2 \in \la a,b\ra$. Note that $k_1,k_2 \in [a-1]$.
	%		We now prove that any element $s \in \la a,b,c \ra$ can be uniquely written as $ax+by$ or $ax+cz$ where
	%		$x \in \NN, y \in [k_1-1], z \in [k_2-1]\setminus\{0\}$.
	%		Let $c k_2 = aj + bk$ for some $j,k \in \NN$. 
	%		Clearly, any $s=ax_1+bx_2+cx_3$ can be written in the form $s=a x + bv+cw$ for some $v,w \in \NN_{\leqslant a-1}$. Since $a \mid b+c$, this can be reduced to a representation of the form $s= ax + bu$ or $s=a x + cu$ for some $u \in [a-1] $. For the sake of brevity, assume $s= ax+cu$ where $u >k_2-1$. 
	%		Let $i \in \NN$ be such that $0\leqslant u-ik_2<k_2$. We now have two cases: $ik \leqslant u-ik_2$ and $u-ik_2<ik$. In the first case we can write 
	%		\[s= a\left(x+ij\right)+b\left(ik\right)+c\left(u-ik_2\right)= a\left(x+ij+n_{a,b}ik\right) + c\left(u-ik_2-ik\right).\]
	\noindent
	%		Since, $u-ik_2-ik \in [k_1-1]$, we get a representation of the desired form.
	
	%		In the second case, we have 
	%\[s= a\left(x+ij\right)+b\left(ik\right)+c\left(u-ik_2\right)= a\left(x+ij+n_{a,b}\left(u-ik_2\right)\right) + b\left(i\left(k_2+k\right)-u\right).\]
	
	%	If $i(k_2+k)-u \notin [k_2-1]$, we can equivalently write $s$ as $s=ax(1)+bu(1)$ for some $u \in [a-1]\setminus\{0\}$. If $u(1)\notin [k_2-1]$, we can repeat this process to get a representation of the form $s=ax(2)+cu(2)$. Since the size of the coefficient of $a$ strictly increases at every step, $u(2)<u$. 
	%	Clearly, we can continue this process until we have a representation of the form $ax+by$ or $ax+cz$ where
	%	$x \in \NN, y \in [k_1-1], z \in [k_2-1]\setminus\{0\}$. Now suppose $ax+by= ax'+cz$ for some $y \in [k_1-1], z \in [k_2-1]$. Since $\gcd(b,c) = 1$, we get $x \neq x'$. For the sake of brevity, assume $x<x'$. We now get $by = a(x'-x)+cz \in \la a,c \ra$ which contradicts the minimality of $k_1$. Thus the representation is unique. Since $\la a,b,c \ra$ contains all residue classes modulo $a$, we get $k_1+k_2-1=a$. 
	
	By Proposition \ref{uniq}, any element $s \in \la a,b,c \ra$ occurs exactly once in one of the sequences $\{d_{i,j}\}_{j=1}^{\infty}$ defined by 
	
	\[ d_{i,j} = 
	\begin{cases}
		c\left(a+1-k_1-i\right)+a\left(j-1\right) \,\,\,\,\,\,\,\,\,\,\, i \in [a-k_1]\setminus\{0\} \\
		a\left(j-1\right)  \,\,\,\,\,\,\,\,\,\,\,\,\,\,\,\,\,\,\,\,\,\,\,\,\,\,\,\,\,\,\,\,\,\,\,\,\,\,\,\,\,\,\,\,\,\,\,\quad \,\,\,\, i=a-k_1+1\\
		b\left(i+k_1-a-1\right)+a\left(j-1\right) \,\,\,\,\,\,\,\,\,\,\, i \in \{a-k_1+2,a-k_1+3,\ldots,a\}.
	\end{cases}\]
	
	Thus, similar to the previous section, we can give a tabular representation of $\la a,b,c \ra$ by labelling the square in row $i$ and column $j$  in $R\left(a\right)$ by $d_{i,j}$. Let $R\left(a,b,c\right)$ denote the grid we obtain after labelling $R\left(a\right)$ in this way. For an ideal $I \subseteq I(a,b,c)$, let $d_{k}(I)$ denote the number of elements in the $k^{\text{th}}$ row of $R(a,b,c)$ that are not in $I$, that is $d_{k}(I) = |\{d_{k,i}\}_{i \geqslant 1} \setminus I|$. 
	
	\begin{prop}\label{length}
		Given any ideal $I \in I(a,b,c)$, the following inequalities hold \[d_{1}(I)\leqslant d_{2}(I) \leqslant \ldots \leqslant d_{a-k_1+1}(I), \quad d_{a}(I)\leqslant d_{a-1}(I) \leqslant \ldots \leqslant d_{a-k_1+1}(I). \]
	\end{prop}
	\begin{proof}
		We first prove that $d_{i}(I) \leqslant d_{i+1}(I)$ for $i \in [a-k_1-1]$. Let $d_i(I) = m-1$. Thus, $d_{i,m} \in I$ which implies $d_{i,m} + c = d_{i-1,m} \in I $. Therefore, $d_{i}(I) \leqslant d_{i+1}(I)$. The remaining set of inequalities can be proven analogously.
	\end{proof} 
 \begin{rem}\label{rem1}
     It can be analogously proved that any non-negative integer sequence $\{x_i\}_{i=1}^{a}$ that satisfies
	\[x_{1}\leqslant x_{2} \leqslant \ldots \leqslant x_{a-k_1+1}, \quad x_{a}\leqslant x_{a-1} \leqslant \ldots \leqslant x_{a-k_1+1}, \]
	corresponds to an ideal $I \in I(a,b,c)$ by mapping the sequence $\{x_i\}_{i=1}^{a}$ to $\bigcup_{i=1}^{a} \{d_{i,j}\}_{j > x_j}$
 \end{rem}
	 For some $k \in \NN_{\geqslant 1}$, denote by $R_{k}\left(a,b,c\right)$ the set of ideals $I \subseteq I(a,b,c)$ such that we have $d_{j}(I)\geq k-1$ for all $j \in [a]$ and $d_{1,k} \in I$ or $d_{a,k} \in I$. Since any ideal belongs to exactlt one of the sets $R_i(a,b,c)$, we obtain the set partition $I(a,b,c) = \bigcup_{i \in \NN_{\geqslant 1}} R_i(a,b,c)$. We now state one of our main results.
	\begin{thm}\label{thma}
		The ideal generating function for the semigroup $\langle a,b,c \rangle$ is given by $$I\left(a,b,c;q\right) = \dfrac{1}{1-q^a}\left(\sum_{I\,\in \,R_1(a,b,c)} {q^{\co(I)}}\right).$$
	\end{thm}
	\begin{proof} For some positive integer $ m \in \NN_{\geqslant 1}$, let $$F: R_1\left(a,b,c\right) \mapsto R_m\left(a,b,c\right), \quad G: R_m\left(a,b,c\right) \mapsto R_1\left(a,b,c\right),$$ be maps defined by $F(I)= I+a(m-1)$ and $G(I)= I-a(m-1)$ respectively. We first prove that $F$ is well defined. Clearly, for $I \in I(a,b,c)$ we have $I+a \in I(a,b,c)$. Thus $I+a(m-1)$ is an ideal. Since $d_{1,1} \in I$ or $d_{a,1} \in I$ we have $d_{1,1}+a(m-1) = d_{1,m} \in I+a(m-1)$ or $d_{a,1}+a(m-1) = d_{a,m} \in I+a(m-1)$. Thus $I+a(m-1) \in R_{m}(a,b,c)$. One can analogously prove $G$ is well-defined. Clearly $F,G$ are injective implying $|R_1(a,b,c)|=|R_m(a,b,c)|$. Now let 
		\[S_i(q) = \sum_{I \in R_i(a,b,c)} q^{\co(I)}.\]
		From the bijection between $R_1(a,b,c)$ and $R_m(a,b,c)$ we deduce that
		\[S_{m}(q) = q^{a(m-1)} S_{1}(q).\]
		Thus, we have 
		\[I(a,b,c;q) = \sum_{m=1}^{\infty} S_m(q) = \frac{S_1(q)}{1-q^a},\]
		as desired.
	\end{proof}
	
	As an application of this result, we now compute $I\left(a,b,c;q\right)$ for a particular parametrized special case. 
	
	\subsection{Computing $I(3,n+2,2n+1)$} Throughout this section we assume that $n \in \NN$ and $n \neq 1 \pmod 3$. We enumerate ideals in $I(3,n+2,2n+1)$ by removing certain elements from the tabular representation of $\la 3,n+2,2n+1 \ra$. Since $2(n+2)=3+(2n+1)$ and $2(2n+1)= 3n + (n+2)$, we get that any element $s \in \langle 3,n+2,2n+1 \rangle$ occurs exactly once in one of the sequences $\{d_{1,j}\}_{j=0}^{\infty},\{d_{2,j}\}_{j=1}^{\infty},\{d_{3,j}\}_{j=1}^{\infty}$ that satisfy $d_{i,j+1}=d_{i,j}+3$ and begin with $d_{1,1}=2n+1$, $d_{2,1}=3$, $d_{3,1}=n+2$.  Figure \ref{fig:ffig} shows the tabular representation of $\la 3,n+2,2n+1 \ra$.

	We can partition the set $ R_1(3,n+2,2n+1)$ as $\la 3,n+2,2n+1 \ra \cup P_1 \cup P_2 \cup P_3$ where 
	\begin{align*}
		&P_1= \{I \in R_1\left(3,n+2,2n+1\right) \mid 0 \notin I, n+2 \in I, 2n+1 \in I\}, \\ &P_2= \{I \in R_1\left(3,n+2,2n+1\right) \mid 0 \notin I, n+2 \in I, 2n+1 \notin I\}, \\& P_3= \{I \in R_1\left(3,n+2,2n+1\right) \mid 0 \notin I, n+2 \notin I, 2n+1 \in I\}.
	\end{align*} Denote by $S\left(i_1,i_2,i_3\right)=\{d_{1,i}\}_{i \geqslant i_1} \cup \{d_{2,i}\}_{i \geqslant i_1} \cup \{d_{3,i}\}_{i \geqslant i_3}$. For $i \in \{1,2,3\}$, let \[F\left(P_i, q\right) = \sum_{I \in P_i} q^{\co\left(I\right)}.\]
	
	Since for any ideal $I \in R_1(3,n+2,2n+1)$, either $n+2 \in I$ or $2n+1 \in I$ we get that $n+2+2n+1=3n+3 \in I$. Thus for any ideal $I \in R_1$, we have $ d_{2}(I) \in  [n+1]\setminus\{0\}$. Clearly, any set of the form $S(0,i,0)$ where $i \in [n+1]\setminus\{0\}$ is an ideal. From Proposition \ref{length} we have $P_1 = \{S(0,i,0) \mid i\in [n+1]\setminus\{0\}\}$. Thus, we have 
	\begin{equation} \label{1}
		F\left(P_1,q\right) = \sum_{i=1}^{n+1} q^{i}.
	\end{equation}
	
	Next, we calculate $F(P_2,q)$. Since for an ideal $I \in P_2$, $n+2 \in I$ we have that $ (n+2)+(n+2)=2n+4 \in I $. This implies that $d_{1}(I)=1$.  Clearly, any set of the form $S(1,i,0)$ where $i \in [n+1]\setminus\{0\}$ is an ideal. From Proposition \ref{length} we have $P_2 = \{S(1,i,0) \mid i \in [n+1]\setminus\{0\} \}$. Thus, we deduce that 
	\begin{equation} \label{2}
		F\left(P_2,q\right) = \sum_{i=1}^{n+1} q^{i+1}.
	\end{equation}
	
	Finally we calculate $F(P_3,q)$. Since $2n+1 \in I$ for any ideal $I \in P_3$,  the inclusion $ (2n+1)+(2n+1)=4n+2 \in I $ holds. This implies that $d_{3}(I) \in [n]\setminus\{0\}$. From Proposition \ref{length}, we conclude that any set of the form $S(0,i,j), S(0,n+1,k)$ where $1<j<i$, $i \in [n]\setminus\{0\}$ , and $k \in [n]\setminus\{0\}$ is an ideal. Thus we have $$P_3= \left\{S\left(0,i,j\right) \mid i \in [n]\setminus\{0\}, j\in[i]\setminus\{0\}\right\}\cup \left\{S\left(0,n+1,i\right) \mid i \in [n]\setminus\{0\}\right\}.$$ Thus, we have 
	\begin{equation} \label{3}
		F\left(P_3,q\right) = \sum_{i=1}^{n} q^{i} \sum_{j=1}^{i} q^{j} + q^{n+1}\sum_{i=1}^{n}q^i.
	\end{equation}
	Using Theorem \ref{thma} and combining \eqref{1}, \eqref{2}, \eqref{3} gives us the following result.
	\begin{figure}
		\centering

		\begin{tikzpicture}
			
			\begin{scope}[xshift=0cm]
				\foreach \i in {1,...,6} {
					\foreach \j in {1,...,3} {
						\draw (\i, \j) rectangle (\i+1, \j+1);
					}
				}
				
				\node at (1.5,1.5) {\scalebox{0.8}{$2n+1$}};
				\node at (1.5,2.5) {\scalebox{0.9}{$0$}};
				\node at (1.5,3.5) {\scalebox{0.9}{$n+2$}};
				\node at (2.5,1.5) {\scalebox{0.8}{$2n+4$}};
				\node at (2.5,2.5) {\scalebox{0.9}{$3$}};
				\node at (2.5,3.5) {\scalebox{0.9}{$n+5$}};
				\node at (3.5,1.5) {$\ldots$};
				\node at (3.5,2.5) {$\ldots$};
				\node at (3.5,3.5) {$\ldots$};
				\node at (4.5,1.5) {\scalebox{0.8}{$5n+3$}};
				\node at (4.5,2.5) {\scalebox{0.8}{$3n$}};
				\node at (4.5,3.5) {\scalebox{0.8}{$4n+2$}};
				\node at (5.5,1.5) {\scalebox{0.8}{$5n+6$}};
				\node at (5.5,2.5) {\scalebox{0.8}{$3n+3$}};
				\node at (5.5,3.5) {\scalebox{0.8}{$4n+5$}};
				\node at (6.5,1.5) {\scalebox{1}{$\ldots$}};
				\node at (6.5,2.5) {\scalebox{1}{$\ldots$}};
				\node at (6.5,3.5) {\scalebox{1}{$\ldots$}};
			\end{scope}
		\end{tikzpicture}
		
		\caption{ Tabular representation of $R(3,n+2,2n+1)$.}

		\label{fig:ffig}
		
	\end{figure}

	\begin{thm}\label{3,n+2}
		The ideal generating function for the numerical semigroup $\langle 3,n+2,2n+1\rangle$ is given by
		\begin{align}
			I\left(3,n+2,2n+1;q\right) =	\frac{1}{1-q^3} \left( \frac{(q^{n+1}-1)(q^{n+3}+q^{n+2}-q^{n+1}+q^4-q^2-1)}{\left(1-q\right)^2\left(1+q\right)}\right).\nonumber
		\end{align}
	\end{thm}
	
	It can be easily calculated that $\deg(I(3,n+2,2n+1;q)) = 2(n-1)$. We now use Theorem \ref{0} to deduce that $F\left(3,n+2,2n+1\right)=2n-2$ for all $n \not\equiv 1 \pmod 3$. 
	Moreover, Theorem \ref{0} also implies 
	\begin{equation}
		\frac{(q^{n+1}-1)(q^{n+3}+q^{n+2}-q^{n+1}+q^4-q^2-1)}{\left(1-q\right)^2\left(1+q\right)} = (1+q+q^2)f(q),\nonumber
	\end{equation}
	for some polynomial $f(q) \in \ZZ[q]$. However, since the polynomial 
	\[P\left(q\right)=\frac{(q^{n+1}-1)(q^{n+3}+q^{n+2}-q^{n+1}+q^4-q^2-1)}{\left(1-q\right)^2\left(1+q\right)},\]
	is unimodal and $P(q) \in \NN[q]$, we get that $f(q)\in \NN[q]$. This proves Conjecture \ref{conjros} for all numerical semigroups of the form $\la 3,n+2,2n+1 \ra$ with $n \not\equiv 1 \pmod 3$. Moreover, we also have 
	\[m\left(\la 3,n+2,2n+1 \ra , 2(n-1)\right) = \frac{n^2+7n+6}{6}.\]

\section{Counting Ideals with Respect to Conductor}

	\begin{figure}
		
		\begin{center}
			
			\begin{tikzpicture}
				\draw (1,1) rectangle (2,2);
				\draw (2,1) rectangle (3,2);
				\draw (2,2) rectangle (3,3);
				\draw (3,1) rectangle (4,2);
				\draw (3,2) rectangle (4,3);
				\draw (3,3) rectangle (4,4);
				\node at (1.5,1.5) {$0$};
				\node at (2.5,1.5) {$3$};
				\node at (3.5,1.5) {$6$};
				\node at (2.5,2.5) {$4$};
				\node at (3.5,2.5) {$7$};
				\node at (3.5,3.5) {$8$};

				\begin{scope}[xshift=3cm]
					\foreach \i in {1,...,3} {
						\foreach \j in {1,...,3} {
							\draw (\i, \j) rectangle (\i+1, \j+1);
						}
					}
					\node at (1.5,1.5) {$9$};
					\node at (2.5,1.5) {$12$};
					\node at (3.5,1.5) {$\ldots$};
					\node at (1.5,2.5) {$10$};
					\node at (1.5,3.5) {$11$};
					\node at (2.5,2.5) {$13$};
					\node at (3.5,2.5) {$\ldots$};
					\node at (2.5,3.5) {$14$};
					\node at (3.5,3.5) {$\ldots$};
				\end{scope}	
			\end{tikzpicture}
		\end{center}
		
		\caption{Grid associated to $\mathcal{F}_3$.}
		\label{fig:fig:fig}
	\end{figure}
	
	This section is devoted towards studying the sequence $n(\mathcal{F}_a,m_{\mathcal{F}_a})$ where $a \in \NN_{\geqslant 2}$ and $\mathcal{F}_a = \langle a,a+1 \rangle$. To every numerical semigroup $\mathcal{F}_a$ we associate the grid-like structure as follows:
	
	\noindent
	Let $C(a)$ denote the grid we obtain after removing the contiguous block of $(i-1)$ squares starting from the first square from row $i$ of $R(a)$. We now label squares in $C(a)$ with entries from $\mathcal{F}_{a}$. Label the square in row $i$ and column $j$ of $C(a)$ with $(a+1)(i-1)+a(j-i)$. Let $L(a)$ denote the grid we obtain after labelling $C(a)$ in this way. Clearly any element $s \in \mathcal{F}_{a}$ occurs exactly once as an entry in the grid $L(a)$. Figure \ref{fig:fig:fig} shows $L(3)$.  Let $L(i,j)$ denote the label of the square in row $i$ and column $j$ of $L(a)$. For $i \in 1+[a-2]$ and $j \in \NN_{\geqslant a}$, an easy calculation yields
	\begin{equation}\label{+1}
		L(i,j)+1=L(i+1,j) , L(a,j)+1=L(1,j+1).
	\end{equation}
		From the definition of ideal we deduce \begin{equation}\label{r1}
		L(i,j) \notin I \implies L(i,j-1),L(i-1,j-1) \notin I.
	\end{equation} 
	For $i \in [a] \setminus\{0\}$ and $i \leq j$, let $\mathcal{L}(i,j) = \bigcup_{k \in i+[j-i]}L_{i,k}$. Since $L(i,j) = L(i,j-1)+a$ for $j>i$, repetitively using \eqref{r1} produces 
	\begin{equation}\label{r3}
		L(i,j) \notin I \implies \bigcup_{k \in [i-1]}\mathcal{L}(i-k,j-k)\notin I.
	\end{equation}
	Since $L(1,a+k+1) = L(a,a+k-1)+a+1$, a similar procedure yields
	\begin{equation}\label{r2}
		L(1,a+k+1) \notin I \implies \bigcup_{i \in [a-1]}\mathcal{L}(a-i,a+k-i-1) \notin I.
	\end{equation}
	Conversely, it's easy to see that any set $\mathcal{L} = \bigcup_{i\in I, j \in J} L_{i,j}$ of labels that satisfies equations \eqref{r3}, \eqref{r2} is an ideal. For any ideal $I \in I(\mathcal{F}_a)$ and an integer $i \in [a]$, let $r(i,I) = \max\{k \mid \mathcal{L}(i,k) \notin I\}-i+1$ where we set $r(i,I)=0$ if $L_{i,i} \in I$. Any ideal can be uniquely recovered from the vector $(r(1,I),r(2,I),\ldots,r(a,I))$. We now state our main result.
	\begin{thm}\label{cn}
		Let $C_n$ denote the $n^{\text{th}}$ Catalan number defined by 
		\[C_n \coloneqq \frac{1}{2n+1} \binom{2n}{n}.\]
		Then for all integers $m \in \NN$, we have 
		\[n\left(\mathcal{F}_{a} , a^2 + \left(m-1\right)a\right) = C_{a}.\]
	\end{thm}
	\begin{rem}
		Since $F\left(\mathcal{F}_a\right) = a^2-a-1$, Theorem \ref{cn} combined with Proposition \ref{ms} implies that $n\left(\mathcal{F}_a , n\right) = C_{a}$ for all integers $2a^2-2a\leqslant n$.	
	\end{rem}
	
	In order to prove Theorem \ref{cn} we need a few auxiliary results. A Dyck path of order $n$ is a lattice path that uses the steps of the form $(1,0),(0,1)$, begins at $(0,0)$, ends at $(n,n)$ and strictly stays above the line $y=x$. Let $D_n$ denote the set of Dyck paths of order $n$ and $S(n)$ denote the squre grid having $n$ rows and $n$ columns. It is known that $|D_n|=C_n$. For any path $\pi \in D_n$, let $a_i(\pi)$ denote the number of squares in row $i$ of $S(n)$ that are to the right of $\pi$ and to the left of the line $y=x$. The statistic $\operatorname{area}(\pi)$ is defined as $\sum_{i=1}^{n} a(i)$. Carlitz and Riordan define a $q$-analog of Catalan numbers with respect to the statistic $\operatorname{area}$ as \[C_n(q) = \sum_{\pi \in D_n} q^{\operatorname{area}(\pi)}.\]
	Let $p_a(n)$ denote the number of vectors  with $a$ components and with positive integer entries that satisfy the following conditions 
	\begin{enumerate}
		\item $p_i \in [a-i+1]$,
		\item $p_{i} \geqslant p_{i+1}$,
		\item $\sum_{i=1}^{a} p_{i} = n.$
	\end{enumerate}
	Let $P_a = \{(p_1,p_2,\ldots,p_a) \mid  p_i \in [a-i+1] , p_{i} \geqslant p_{i+1}\}$. Clearly, $|P_a| = \sum_{n=1}^{\frac{a(a+1)}{2}} p_a(n)$.
	
	\begin{prop}\label{Pa}
		The generating function for the sequence $p_a(n)$ is 
		\[\sum_{n=1}^{\frac{a(a+1)}{2}}p_{a}\left(n\right)q^n = q^{\binom{a+1}{2}} C_{a+1} \left(\frac{1}{q}\right).\]
	\end{prop}
	\begin{proof}
		For any Dyck path $\pi \in D_{a+1}$, let $r(\pi,i) = i-1-a(i)$. Equivalently, $r(\pi,i)$ is the number of squares in row $i$ of $S(a)$ that are to the left of $\pi$. Since $\pi$ is composed of steps of the form $(1,0),(0,1)$, we get $r(\pi,i+1) \geqslant r(\pi,i)$. Since $a(i) \in [i-1]$, the map $F: D_{a+1} \to P_a$ defined by $F(\pi) = (r(\pi,a+1),r(\pi,a),\ldots,r(\pi,2))$ is well defined and injective.
		It can be similarly be proved that any path $\pi \in D_{a+1}$ can be uniquely determined from the vector $(0,p_a,p_{a-1},\ldots,p_{1})$, where $(p_1,p_2,\ldots,p_{a}) \in P_a$. Thus the map $F$ is bijective. Translating the bijection into generating function gives us the desired result.
	\end{proof}

	\begin{ex}
		The five Dyck paths of order $3$ in Figure \ref{fig:5d} correspond to the vectors $(0,0)$, $(0,1)$, $(0,2)$, $(1,1)$, $(1,2)$ under the map $F: D_3 \to P_2$. This implies that $p_2(0)=1,p_2(1)=1,p_2(2)=2,p_2(3)=1$. It can be calculated that $C_3(q)=1+2q+q^2+q^3$. Thus we have \[q^{\binom{3}{2}}\left(1+\frac{2}{q}+\frac{1}{q^2}+\frac{1}{q^3}\right)=q^3+2q^2+q+1,\]
		which is in accordance with the calculation above.
	\end{ex}
Moreover, Proposition \ref{Pa} implies that $|P_a|= C_{a+1}$. We are now ready to prove Theorem \ref{cn}.
	
	\subsection{Proof of Theorem \ref{cn}} We first prove that $n(\mathcal{F}_a, a^2-a) = C_{a+1}$. Let $I \in  \mathcal{A} (\mathcal{F}_{a}, a^2-a)$. From the definition of conductor, we get $\NN_{\geq a^2-a} \subseteq I$. On combining equation \eqref{+1} and the fact that $L(1,a)=a^2-a$, we get $r(I,i) \in [a-i]$ for $i \in 1+[a-2]$. Also for $I$ to be an ideal, the excluded labels must satisfy equation \eqref{r3}. This implies $r(I,i+1) \leq r(I,i)$. Thus $(r(I,1),\ldots,r(I,a-1)) \in P_{a-1}$. This implies that any ideal $I \in \mathcal{A} (\mathcal{F}_{a}, a^2-a)$ corresponds to a vector in $P_{a-1}$. For some vector $\textbf{p}= (p_1,\ldots,p_{a-1}) \in P_{a-1}$, let $p_a=0$ and let \[L(\textbf{p}) = \bigcup_{i \in 1+[a-1], j \in \NN_{\geq j+p_j}} L_{i,j}. \]
Clearly for any element $l \in L(p)$, the inclusion $l+a \in L(p)$ holds. Also since $p_{i} \geq p_{i+1}$, we get $l+a+1 \in L(p)$. Thus the set $L(p)$ is ideal. Since $p_1 \leq a-1$, we get $a^2-a \in L(P)$. Thus $c(L(p))= a^2-a$ and $L(p) \in \mathcal{A}(\mathcal{F}_{a}, a^2-a)$. This proves that vectors in $P_{a-1}$ are in one-to-one correspondence with ideals in $\mathcal{A}(\mathcal{F}_{a}, a^2-a)$. Proposition \ref{Pa} implies $|\mathcal{A} (\mathcal{F}_{a}, a^2-a)|=|P_{a-1}|= C_a$.

	Now, let $I \in \mathcal{A}(\mathcal{F}_a, a^2+(m-1)a )$ for some $m \in \NN_{\geqslant 1}$.  Notice that $L(1,a+m)=a^2+(m-1)a$ and $a^2+(m-1)a-1 = L_{a,a+m-1}$. From the definition of conductor we get $\NN_{\geq a^2+(m-1)a} \subseteq I$ and $L_{a,a+m-1} \notin I$. Equations \eqref{+1},\eqref{r3} imply that any for any ideal $I \in \mathcal{A}(\mathcal{F}_a, a^2+(m-1)a)$, the inclusion $r(I,i) \in m+[a-i]$ holds. However, the removed elements must also satisfy equation \eqref{r3}. Thus we get $r(I,i+1) \leq r(I,i)$. On subtracting $m$ from each entry of the vector $(r(I,1),r(I,2),\ldots,r(I,a-1),r(I,a))$, we conclude that any ideal $I \in \mathcal{A} (\mathcal{F}_{a}, a^2+(m-1)a)$ corresponds to a vector in $P_{a-1}$. It can be analogoulsy proved as before that any vector $p \in P_{a-1}$ corresponds to an ideal in $\mathcal{A} (\mathcal{F}_{k}, a^2+(m-1)a)$. Proposition \ref{Pa} implies $|\mathcal{A} (\mathcal{F}_{a}, a^2-a)|=|P_{a-1}|= C_a$. 
 
 \QED

  \begin{rem}
It can be observed that the number of ideals $I \in  \mathcal{F}_{a}$ such that $c(I)=a^2-a$ and $\operatorname{codim}(I)=k$ is 
\[[q^k]  \left(q^{\binom{a}{2}} C_{a} \left(\frac{1}{q}\right)\right), \]
where $[q^k] F(q)$ denotes the coefficient of $q^k$ in the Taylor series of $F(q)$ centered at $q=0$.
 \end{rem}

	\begin{figure}
		
		\begin{tikzpicture}
			
			\begin{scope}[xshift=-6.4cm]
				\scalebox{0.9}{
					\foreach \i in {1,...,3} {
						\foreach \j in {1,...,3} {
							\draw (\i, \j) rectangle (\i+1, \j+1);
						}
					}
					\draw [line width=0.7mm, black ] (1,1) -- (1,2);
					\draw [line width=0.7mm, black ] (1,2) -- (1,3);
					\draw [line width=0.7mm, black ] (1,3) -- (1,4);
					\draw [line width=0.7mm, black ] (1,4) -- (2,4);
					\draw [line width=0.7mm, black ] (2,4) -- (3,4);
					\draw [line width=0.7mm, black ] (3,4) -- (4,4);}
				\begin{scope}[xshift=3.5cm]
					\scalebox{0.9}{\foreach \i in {1,...,3} {
							\foreach \j in {1,...,3} {
								\draw (\i, \j) rectangle (\i+1, \j+1);
							}
						}
						\draw [line width=0.7mm, black ] (1,1) -- (1,2);
						\draw [line width=0.7mm, black ] (1,2) -- (1,3);
						\draw [line width=0.7mm, black ] (1,3) -- (2,3);
						\draw [line width=0.7mm, black ] (2,3) -- (2,4);
						\draw [line width=0.7mm, black ] (2,4) -- (3,4);
						\draw [line width=0.7mm, black ] (3,4) -- (4,4);}
				\end{scope}
				\begin{scope}[xshift=7cm]\scalebox{0.9}{
						\foreach \i in {1,...,3} {
							\foreach \j in {1,...,3} {
								\draw (\i, \j) rectangle (\i+1, \j+1);
							}
						}
						\draw [line width=0.7mm, black ] (1,1) -- (1,2);
						\draw [line width=0.7mm, black ] (1,2) -- (1,3);
						\draw [line width=0.7mm, black ] (1,3) -- (2,3);
						\draw [line width=0.7mm, black ] (2,3) -- (3,3);
						\draw [line width=0.7mm, black ] (3,3) -- (3,4);
						\draw [line width=0.7mm, black ] (3,4) -- (4,4);}
				\end{scope}
				\begin{scope}[xshift=10.5cm]
					\scalebox{0.9}{
						\foreach \i in {1,...,3} {
							\foreach \j in {1,...,3} {
								\draw (\i, \j) rectangle (\i+1, \j+1);
							}
						}
						\draw [line width=0.7mm, black ] (1,1) -- (1,2);
						\draw [line width=0.7mm, black ] (1,2) -- (2,2);
						\draw [line width=0.7mm, black ] (2,2) -- (2,3);
						\draw [line width=0.7mm, black ] (2,3) -- (2,4);
						\draw [line width=0.7mm, black ] (2,4) -- (3,4);
						\draw [line width=0.7mm, black ] (3,4) -- (4,4);}
				\end{scope}
				\begin{scope}[xshift=14cm]
					\scalebox{0.9}{\foreach \i in {1,...,3} {
							\foreach \j in {1,...,3} {
								\draw (\i, \j) rectangle (\i+1, \j+1);
							}
						}
						\draw [line width=0.7mm, black ] (1,1) -- (1,2);
						\draw [line width=0.7mm, black ] (1,2) -- (2,2);
						\draw [line width=0.7mm, black ] (2,2) -- (2,3);
						\draw [line width=0.7mm, black ] (2,3) -- (3,3);
						\draw [line width=0.7mm, black ] (3,3) -- (3,4);
						\draw [line width=0.7mm, black ] (3,4) -- (4,4);}
				\end{scope}
			\end{scope}
		\end{tikzpicture}
		\caption{Dyck paths of order three}
		\label{fig:5d}
	\end{figure}

	\section{Concluding Remarks and Further Research}\label{S5}
	
	Several paths have not been explored yet and those will be the subject of future work. We now highlight a connection with HOMFLY polynomials. Let $P(L)$ denote the HOMFLY polynomial of an oriented link $L \in S^3$. It is known that $P(L)$ is an element of $\ZZ[a^{\pm 1}, (q+q^{-1})^{\pm 1}]$. Let $T_{k,n}$ denote the $(k,n)$ Torus knot. Given a pair of coprime $(k,n)$ and the curve $y^k=x^n$, we know that $T_{k,n}$ is the link of its singularity at the origin. Oblomkov and Shende prove the following \cite{oblomkov2018hilbert}.
	\begin{thm}\label{oblom}
		Let $C$ be the curve cut out by $y^k=x^n$ and let $p$ be the origin. Then, we have 
		\begin{equation}\label{integral}
			P\left(T_{k,n}\right) = \left(\frac{a}{q}\right)^{(k-1)(n-1)} (1-q^2) \int_{C_{p}^{[*]}} q^{2l} \left(1-a^2\right)^{m-1} d\chi.  
		\end{equation}
	\end{thm}
	Note that the complete local ring associated with the curve $y^k=x^n$ is $\mathcal{O}=\CC[[t^k,t^n]]$. Let $J(k,n)$ denote the set of all monomial ideals in $\CC[[t^k,t^n]]$. The integral in equation \eqref{integral} can alternatively be written as \cite{oblomkov2018hilbert}
	\begin{equation}\label{eqn}
		\int_{C_{p}^{[*]}} q^{2l} (1-a^2)^{m-1} d\chi = \sum_{J\in J(k,n)} q^{2 \dim_{\CC} (\mathcal{O}/J)} (1-a^2)^{m(J)}.
	\end{equation}
	From equations \eqref{firsteqn}, \eqref{eqn}, and Theorem \ref{oblom} we have 
	\begin{equation}\label{Hom}
		P\left(T_{k,n}\right) \mid_{a=0} \,= I\left(k,n;q^2\right).
	\end{equation}
	We wonder whether such relations exist for the ideal generating function of numerical semigroups with more than two generators. Another possible future direction is to study the stationary terms of the sequences $m(S,k)$ and $n(S,k)$. As noted before, $m\left(\la a,b \ra, k\right) = c_{a,b}$ for all $k \geqslant ab-a-b$ are the rational Catalan numbers and $n(\mathcal{F}_a,a^2-a)$ are the Catalan numbers which have been widely studied and have a variety of combinatorial interpretations. Thus, it will be interesting from a combinatorial point of view to study the stationary terms of $m\left(S,k\right),n(S,k)$ for numerical semigroups with more than two generators. Moreover, we have the equality $m(\mathcal{F}_a, a^2-a) = n(\mathcal{F}_{a}, a^2-a)$. This suggests the existence of bijection between the sets $\mathcal{A}(\mathcal{F}_a, a^2-a)$ and $\{I \in I(\mathcal{F}_a) \mid \co(I)=a^2-a\}$. We conclude this paper by inviting the interested reader to find the bijection. 

 \section{Acknowledgements}
	This work was partially done during the Research Science Institute (RSI) at MIT in the summer of 2022. I would like to thank my mentor Jeffery Yu for his guidance throughout the duration of the program, and Minh-Tam Quang Trinh for proposing the project. I am grateful to the RSI, CEE \& MIT for their hospitality and support during the preparation of this work.


\begin{thebibliography}{10}
		
		\bibitem{Geo} G. Andrews, \emph{The Theory of Partitions}. Encyclopedia of Mathematics and its Applications 2.
		Addison-Wesley Publishing Co., Reading, Mass.-London-Amsterdam, 1976.

  		
		\bibitem{baek}  J. Backelin. \emph{On the number of semigroups of natural numbers}. Mathematica Scandinavica, 66 (1990), 197–215.
		
		\bibitem{w1} A. Beauville, Counting rational curves on K3 surfaces. Duke Math. J. 97, 99–108 (1999)
		
		\bibitem{bria} J. Briancon, \emph{Description de $\operatorname{Hilb}^n\CC\{x,y\}$}, Invent. Math. \textbf{41} (1977), 45–89.
		
		\bibitem{qcat} L. Carlitz and J. Riordan, \emph{Two element lattice permutation numbers and their $q$-generalization}, Duke Math. J. 31 (3) 371 - 388, 1964.
		
		
		\bibitem{w2} B. Fantechi, L. Göttsche, D. van Straten, \emph{Euler number of the compactified Jacobian and multiplicity
		of rational curves}. J. Algebr. Geom. 8, 115–133 (1999).

  				\bibitem{Rosa}P.A. Garc\'ia-S\'anchez and J.C. Rosales, \emph{Numerical Semigroups}, New York: Springer, 2009.
		
		\bibitem{hilbn} A. Iarrobino. \emph{Punctual Hilbert schemes}. Mem. AMS \textbf{188}, 1977.
		
		
		\bibitem{RosaMor} M. A.	Moreno-Fr\'ias and J. C. Rosales. \emph{Counting the ideals with given genus of a numerical semigroup}. Journal of Algebra and Its Applications (2022): 2330002.
		
		\bibitem{Rosa2} M.A. Moreno-Frías and J.C. Rosales. \emph{Counting the Ideals with a Given Genus of a Numerical Semigroup with Multiplicity Two}. Symmetry 2021, \textbf{13}, 794.
		
		\bibitem{w4} J.J. Moyano-Fernández, J. Uliczka, \emph{Lattice paths with given number of turns and semimodules over numerical semigroups.} Semigroup Forum 88, 631–646 (2014).
		
		
		
		
		\bibitem{oblomkov2018hilbert}A. Oblomkov and V. Shende. \emph{The Hilbert scheme of a plane curve singularity and the HOMFLY polynomial of its link}. Duke Math. J. \textbf{161} (7) 1277--1303, 2012.
		
		\bibitem{w3} J. Piontkowski, \emph{Topology of the compactified Jacobians of singular curves}. Math. Zeit. \textbf{255} (1), 195–226 (2007).
		
		

		
		
		
		
		
		
		\bibitem{unimod} J. J. Sylvester, \emph{Proof of the hitherto undemonstrated fundamental theorem of invariants}. In The collected mathematical papers of James Joseph Sylvester, Vol. \textbf{3} Cambridge University Press, Chelsea,
		New York (1973), 117--126 
		

		

		
	\end{thebibliography}
\end{document}